\documentclass[12pt, oneside, reqno]{article}
\usepackage{amsmath,amsthm}
\usepackage{amssymb,latexsym}
\usepackage{enumerate}
\usepackage{cases}

\usepackage{tikz}
\newtheorem{thm}{Theorem}[section]

\newtheorem{lem}[thm]{Lemma}

\theoremstyle{definition}

\numberwithin{equation}{section}

\begin{document}

\title{\Large On an inverse problem in additive number theory
}\author{\large Min Tang\thanks{Corresponding author. This work was supported by the National Natural Science Foundation of China(Grant
No. 11971033) and top talents project of Anhui Department of Education(Grant No. gxbjZD05).} and Hongwei Xu}
\date{} \maketitle
 \vskip -3cm
\begin{center}
\vskip -1cm { \small
\begin{center}
 School of Mathematics and Statistics, Anhui Normal
University
\end{center}
\begin{center}
Wuhu 241002, PR China
\end{center}}
\end{center}

  {\bf Abstract:} For a set $A$, let $P(A)$ be the set of all finite subset sums of $A$. In this paper, for a sequence of integers $B=\{1<b_1<b_2<\cdots\}$ and $3b_1+5\leq b_2\leq 6b_1+10$, we determine the critical value for $b_3$ such that there exists an infinite sequence $A$ of positive integers for which
 $P(A)=\mathbb{N}\setminus B$. This result shows that we partially solve the problem of Fang and Fang [`On an inverse problem in additive number theory', Acta Math. Hungar. 158(2019), 36-39].


\vskip8mm

\section{Introduction}
Let $\mathbb{N}$ be the set of all nonnegative integers. For a sequence of integers $A=\{a_1<a_2<\cdots\}$, let
$$P(A)=\left\{\sum \varepsilon_ia_i: a_i\in A, \varepsilon_i=0\text{ or }1, \sum \varepsilon_i<\infty\right\}.$$
Here $0\in P(A)$.
In 1970, Burr \cite{Burr} asked the following question: which subsets $S$ of $\mathbb{N}$ are equal to $P(A)$ for some $A$?
Burr showed the following result (unpublished):

\noindent{\bf Theorem A} (\cite{Burr}).
{\it Let $B=\{4\leq b_1<b_2<\cdots\}$ be a sequence of integers for which $b_{n+1}\geq b_n^2$ for $n=1,2,\ldots$. Then there exists $A=\{a_1<a_2<\cdots\}$ for which
 $P(A)=\mathbb{N}\setminus B$.}

 Burr \cite{Burr} ever mentioned that if $B$ grows sufficiently rapidly, then there exists a sequence $A$ such that  $P(A)=\mathbb{N}\setminus B$.
 In 1996, Hegyv\'{a}ri \cite{Hegy} showed some examples that if sequence $B$ grow slowly, then there is no sequence $A$ for which
 $P(A)=\mathbb{N}\setminus B$.

 \noindent{\bf Theorem B} (\cite{Hegy}, Theorem 2).
{\it Let $B=\{b_1<b_2<\cdots\}$ be a sequence of integers. Assume that for $n>n_0$, $b_{n+1}/b_n\leq \sqrt{2}$ holds and $B$ is a Sidon sequence, i.e
$b_i+b_j=b_k+b_t$ implies $i=k$; $j=t$ or $i=t$; $j=k$. Then there is no sequence $A$ for which
 $P(A)=\mathbb{N}\setminus B$.}

In 2012, Chen and Fang \cite{chen2012} extended Hegyv\'{a}ri's result by elementary but not easy argument. In 2013, Chen and Wu \cite{chen2013} further continued Burr's question.

\noindent{\bf Theorem C} (\cite{chen2013}, Theorem 1).
{\it If $B=\{b_1<b_2<\cdots\}$ is a sequence of integers with $b_1\in\{4,7,8\}\cup \{b: b\geq 11, b\in \mathbb{N}\}$, $b_2\geq 3b_1+5$, $b_3\geq 3b_2+3$ and $b_{n+1}> 3b_n-b_{n-2}$ for all $n\geq 3$, then there exists a sequence of positive integers $A=\{a_1<a_2<\cdots\}$ such that
 $P(A)=\mathbb{N}\setminus B$ and
 $$P(A_s)=[0,2b_s]\setminus \{b_1,\ldots, b_s, 2b_s-b_{s-1}, \ldots,2b_s-b_1\},$$
 where $A_s=A\cap [0, b_s-b_{s-1}]$ for all $s\geq 2$.}

In \cite{chen2013}, Chen and Wu proved that these lower bounds are optimal in a sense. Moreover, they posed the following problem:

  \noindent{\bf Problem 1} (\cite{chen2013}, Problem 1). {\it Let $B=\{b_1<b_2<\cdots\}$ be a sequence of positive integers. Let $d_1=10$, $d_2=3b_1+4$, $d_3=3b_2+2$ and $d_{n+1}=3b_n-b_{n-2}(n\geq 3)$. If $b_m=d_m$ for some $m\geq 3$ and $b_n>d_n$ for all $n\neq m$. Is it true there is no sequence of positive integers $A=\{a_1<a_2<\cdots\}$ with $P(A)=\mathbb{N}\setminus B$?}

Recently, we show that the answer to Problem 1 is negative for $m=3$.

\noindent{\bf Theorem D} (\cite{Tang}, Theorem 1.1). {\it Let $B=\{11\leq b_1<b_2<\cdots\}$ be a sequence of integers with $b_2=3b_1+5$, $b_3=3b_2+2$ and $b_{n+1}=3b_n+4b_{n-1}$ for all $n\geq 3$. Then there exists a sequence of positive integers $A=\{a_1<a_2<\cdots\}$ such that
 $P(A)=\mathbb{N}\setminus B$.}

With the further research of Burr's question, many interesting problems arise. In 2019, for $b_2=3b_1+5$, Fang and Fang \cite{Fang2019} determined the critical value for $b_3$ such that there exists an infinite sequence of positive integers $A$ for which
 $P(A)=\mathbb{N}\setminus B$.

 \noindent{\bf Theorem E} (\cite{Fang2019}, Theorem 1.1).
{\it If $A$ and $B=\{1<b_1<b_2<\cdots\}$ are two infinite sequences of positive integers with $b_2=3b_1+5$ such that
 $P(A)=\mathbb{N}\setminus B$, then $b_3\geq 4b_1+6$. Furthermore, there exist two infinite sequences of positive integers $A$ and $B=\{1<b_1<b_2<\cdots\}$
with $b_2=3b_1+5$ and $b_3=4b_1+6$ such that
 $P(A)=\mathbb{N}\setminus B$.}

 Moreover, Fang and Fang \cite{Fang2019} posed the following problem:

  \noindent{\bf Problem 2} (\cite{Fang2019}, Problem 1.2). {\it For any $b_2\geq 3b_1+5$, determine the critical value for $b_3$ depending only on $b_1$
  and $b_2$ such that there exists an infinite sequence $A$ of positive integers for which $P(A)=\mathbb{N}\setminus B$.}

 In this paper, we partially solve their problem:
 \begin{thm}\label{thm1} Let $A$ and $B=\{1<b_1<b_2<\cdots\}$ be two infinite sequences of positive integers such that $P(A)=\mathbb{N}\setminus B$. If $3b_1+5\leq b_2\leq 6b_1+10$, then $b_3\geq b_2+b_1+1$.
\end{thm}

 \begin{thm}\label{thm2} There exist two infinite sequences of positive integers $A$ and $B=\{b_1<b_2<\cdots\}$ with $b_1\in\{4,7,8\}\cup \{b: b\geq 11, b\in \mathbb{N}\}$, $3b_1+5\leq b_2\leq 6b_1+10$ and $b_3=b_2+b_1+1$ such that $P(A)=\mathbb{N}\setminus B$.
\end{thm}
\section{Lemmas}

\begin{lem}\label{lem1}(\cite{chen2012}, Lemma 1). Let $A=\{a_1<a_2<\cdots\}$ and $B=\{b_1<b_2<\cdots\}$ be two sequences of positive integers with $b_1>1$ such that $P(A)=\mathbb{N}\setminus B$.
Let $a_k<b_1<a_{k+1}$. Then $$P(\{a_1,\ldots,a_i\})=[0,c_i], \quad i=1,2,\ldots,k,$$
where $c_1=1$, $c_2=3$, $c_{i+1}=c_i+a_{i+1}(1\leq i\leq k-1)$, $c_k=b_1-1$ and $c_i+1\geq a_{i+1} (1\leq i\leq k-1)$.
\end{lem}

\begin{lem}\label{lem2} Let $A=\{a_1<a_2<\cdots\}$ and $B=\{1<b_1<b_2<\cdots\}$ be two sequences of positive integers such that $P(A)=\mathbb{N}\setminus B$. Write $P(\{a_1,\ldots,a_k\})=[0,b_1-1]$.
If $b_2\geq 3b_1+5$, then we have

\noindent(i) $a_{k+1}=b_1+1$, $a_{k+2}\leq 2b_1+1$, $a_{k+3}\leq a_{k+2}+b_1$, $b_2\geq a_{k+3}+a_{k+2}+b_1$;

\noindent(ii) $P(\{a_1,\ldots,a_{k+3}\})=[0, a_{k+3}+a_{k+2}+2b_1] \setminus \{b_1, a_{k+3}+a_{k+2}+b_1\}.$
\end{lem}

\begin{proof} Since $b_1>1$, we have $a_1=1$. Let $a_k<b_1<a_{k+1}$. By Lemma \ref{lem1}, we have $$P(\{a_1,\ldots,a_k\})=[0, b_1-1].$$
Since $b_1+1\in P(A)$ and $b_1\not\in P(A)$, we have $a_{k+1}=b_1+1$. Hence $$P(\{a_1,\ldots,a_{k+1}\})=[0, 2b_1]\setminus \{b_1\},$$
and $$a_{k+2}+P(\{a_1,\ldots,a_{k+1}\})=[a_{k+2}, a_{k+2}+2b_1]\setminus \{a_{k+2}+b_1\}.$$
If $a_{k+2}\geq 2b_1+2$, then $2b_1+1\not\in P(A)$ and $b_2=2b_1+1$, a contradiction. So $a_{k+2}\leq 2b_1+1$ and \begin{eqnarray}{\label{eq2.1}}P(\{a_1,\ldots,a_{k+2}\})=[0, a_{k+2}+2b_1] \setminus \{b_1, a_{k+2}+b_1\}.\end{eqnarray}
Then \begin{eqnarray}{\label{eq2.2}}a_{k+3}+P(\{a_1,\ldots,a_{k+2}\})&=&[a_{k+3}, a_{k+3}+a_{k+2}+2b_1]\\
&\setminus &\{a_{k+3}+b_1, a_{k+3}+a_{k+2}+b_1\}.\nonumber\end{eqnarray}
If $a_{k+3}>a_{k+2}+b_1$, then $b_2=a_{k+2}+b_1\leq 3b_1+1$, a contradiction. So $a_{k+3}\leq a_{k+2}+b_1$. Since $$a_{k+3}\leq a_{k+2}+b_1<a_{k+3}+b_1\leq a_{k+2}+2b_1,$$
by (\ref{eq2.1}) and (\ref{eq2.2}) we have
$$P(\{a_1,\ldots,a_{k+3}\})=[0, a_{k+3}+a_{k+2}+2b_1] \setminus \{b_1, a_{k+3}+a_{k+2}+b_1\}.$$
Thus
$$b_2\geq a_{k+3}+a_{k+2}+b_1.$$

This completes the proof of Lemma \ref{lem2}.
\end{proof}

\section{Proof of Theorem \ref{thm1}}
Write $b_2=3b_1+5+m$. Then $0\leq m\leq 3b_1+5$.
By Lemma \ref{lem2}(i), we have
$$3b_1+5+m=b_2\geq a_{k+3}+a_{k+2}+b_1\geq 2a_{k+2}+b_1+1,$$
thus $a_{k+2}\leq b_1+2+\lfloor\frac{m}{2}\rfloor$. Moreover, by Lemma \ref{lem2}(i), we have
 $a_{k+2}\leq 2b_1+1$. Thus
 $$a_{k+2}\leq b_1+2+\min\left\{\left\lfloor\frac{m}{2}\right\rfloor, b_1-1\right\}.$$
By Lemma \ref{lem2}(i) we know that $a_{k+1}=b_1+1$, thus we can write
\begin{equation}\label{eq3.0}a_{k+2}=b_1+2+j, \quad 0\leq j\leq \min\left\{\left\lfloor\frac{m}{2}\right\rfloor, b_1-1\right\}. \end{equation} Again by Lemma \ref{lem2}(i), we have
$$3b_1+5+m=b_2\geq a_{k+3}+a_{k+2}+b_1= a_{k+3}+2b_1+2+j,$$
 thus $$a_{k+3}\leq b_1+3+m-j.$$ Moreover, by Lemma \ref{lem2}(i), we have
 $$a_{k+3}\leq a_{k+2}+b_1=b_1+3+j+b_1-1.$$
For fixed $0\leq j\leq \min\left\{\left\lfloor\frac{m}{2}\right\rfloor, b_1-1\right\}$, let
\begin{equation}\label{eq3.1}a_{k+3}=b_1+3+j+l, \quad 0\leq l\leq \min\{m-2j,b_1-1\}.\end{equation}

By (\ref{eq3.0}), (\ref{eq3.1}) and Lemma \ref{lem2}(ii), we have
\begin{equation}{\label{eq1}}P(\{a_1,\ldots,a_{k+3}\})=[0, 4b_1+5+2j+l] \setminus \{b_1, 3b_1+5+2j+l\}.\end{equation}

If $l+2j=m$, then by (\ref{eq1}) we have
 $$P(\{a_1,\ldots,a_{k+3}\})=[0, 4b_1+5+m] \setminus \{b_1, b_2\}.$$
So, $b_3\geq 4b_1+6+m=b_2+b_1+1$.

If $0\leq 2j+l\leq m-1$, then \begin{equation}\label{eq3.2}3b_1+5+2j+l\leq 3b_1+5+m-1<b_2,\end{equation}
 \begin{eqnarray}{\label{eq2}}a_{k+4}+P(\{a_1,\ldots,a_{k+3}\})&=&[a_{k+4}, a_{k+4}+4b_1+5+2j+l]\\
&&\setminus \{a_{k+4}+b_1, a_{k+4}+3b_1+5+2j+l\}.\nonumber\end{eqnarray}
If $a_{k+4}>3b_1+5+2j+l$, then $b_2=3b_1+5+2j+l$, which contradicts with (\ref{eq3.2}). So $$a_{k+4}\leq 3b_1+5+2j+l.$$

{\bf Case 1:} $a_{k+4}=2b_1+5+2j+l$. Then $a_{k+4}+b_1=3b_1+5+2j+l$, thus
\begin{eqnarray}\nonumber P(\{a_1,\ldots,a_{k+4}\})&=&[0, 6b_1+10+4j+2l]\\
&& \setminus \{b_1, 3b_1+5+2j+l, 5b_1+10+4j+2l\}.\nonumber\end{eqnarray}
Hence \begin{eqnarray}\nonumber a_{k+5}&+&P(\{a_1,\ldots,a_{k+4}\})
=[a_{k+5},a_{k+5}+6b_1+10+4j+2l]\\&&\setminus \{a_{k+5}+b_1,a_{k+5}+3b_1+5+2j+l,a_{k+5}+5b_1+10+4j+2l\}.\nonumber\end{eqnarray}
If $a_{k+5}>3b_1+5+2j+l$, then $b_2=3b_1+5+2j+l,$ which contradicts with (\ref{eq3.2}).
So $$a_{k+5}\leq 3b_1+5+2j+l.$$ Moreover, $a_{k+5}\geq a_{k+4}+1=2b_1+6+2j+l$.
Since $$3b_1+6+2j+l\leq a_{k+5}+b_1\leq 4b_1+5+2j+l,$$ $$5b_1+11+4j+2l\leq a_{k+5}+3b_1+5+2j+l\leq 6b_1+10+4j+2l,$$
we have
$$ P(\{a_1,\ldots,a_{k+5}\})=[0, a_{k+5}+6b_1+10+4j+2l] \setminus \{b_1,a_{k+5}+5b_1+10+4j+2l\}.$$
Hence $$b_2\geq a_{k+5}+5b_1+10+4j+2l>6b_1+10,$$ a contradiction.

{\bf Case 2:}  $a_{k+4}\neq 2b_1+5+2j+l$.
Then by (\ref{eq1}) and (\ref{eq2}), we have
\begin{eqnarray}\label{eq3.60}P(\{a_1,\ldots,a_{k+4}\})&=&[0, a_{k+4}+4b_1+5+2j+l]\nonumber \\
&\setminus &\{b_1, a_{k+4}+3b_1+5+2j+l\}.\end{eqnarray}
Thus $$3b_1+5+m=b_2\geq a_{k+4}+3b_1+5+2j+l.$$
Hence $a_{k+4}\leq m-2j-l$.

If $a_{k+4}=m-2j-l$, then
 $$P(\{a_1,\ldots,a_{k+4}\})=[0, 4b_1+5+m] \setminus \{b_1, b_2\}.$$
So, $b_3\geq 4b_1+6+m=b_2+b_1+1$.

If $a_{k+4}<m-2j-l$, then
\begin{equation}\label{eq3.7}a_{k+4}+3b_1+5+2j+l<3b_1+5+m=b_2.\end{equation}
By (\ref{eq3.60}) we have \begin{eqnarray}\label{eq3.8}a_{k+5}&+&P(\{a_1,\ldots,a_{k+4}\})=[a_{k+5}, a_{k+5}+a_{k+4}+4b_1+5+2j+l]\nonumber \\
&\backslash &\{a_{k+5}+b_1, a_{k+5}+a_{k+4}+3b_1+5+2j+l\}.\end{eqnarray}
If $a_{k+5}>a_{k+4}+3b_1+5+2j+l$, then by (\ref{eq3.60}) and (\ref{eq3.8}), we have $$b_2=a_{k+4}+3b_1+5+2j+l,$$which contradicts with (\ref{eq3.7}).
Thus
 $$a_{k+5}\leq a_{k+4}+3b_1+5+2j+l.$$

{\bf Subcase 2.1:} $a_{k+5}+b_1\neq a_{k+4}+3b_1+5+2j+l$, then by (\ref{eq3.60}) and (\ref{eq3.8}), we have
\begin{eqnarray}\label{eq3.9}P(\{a_1,\ldots,a_{k+5}\})&=&[0, a_{k+5}+a_{k+4}+4b_1+5+2j+l] \nonumber\\
&\setminus &\{b_1, a_{k+5}+a_{k+4}+3b_1+5+2j+l\}.\end{eqnarray}

If $a_{k+5}>m-a_{k+4}-2j-l$, then $$b_2\geq a_{k+5}+a_{k+4}+3b_1+5+2j+l>3b_1+5+m,$$ which contradicts with (\ref{eq3.7}).

If $a_{k+5}=m-a_{k+4}-2j-l$, then $$P(\{a_1,\ldots,a_{k+5}\})=[0, 4b_1+5+m] \setminus \{b_1, b_2\}.$$
So, $b_3\geq 4b_1+6+m=b_2+b_1+1$.

If $a_{k+5}<m-a_{k+4}-2j-l$, then \begin{equation}\label{eq310}a_{k+5}+a_{k+4}+3b_1+5+2j+l<3b_1+5+m=b_2.\end{equation}
By (\ref{eq3.9}) we have\begin{eqnarray}\label{eq3.11}a_{k+6}&+&P(\{a_1,\ldots,a_{k+5}\})=\left[a_{k+6}, \sum_{j=4}^6a_{k+j}+4b_1+5+2j+l\right]\nonumber\\
&&\Bigg\backslash\left\{a_{k+6}+b_1, \sum_{j=4}^6a_{k+j}+3b_1+5+2j+l\right\}.\end{eqnarray}
If $a_{k+6}>a_{k+5}+a_{k+4}+3b_1+5+2j+l$, then by (\ref{eq3.9}) and (\ref{eq3.11}), we have $$b_2=a_{k+5}+a_{k+4}+3b_1+5+2j+l,$$which contradicts with (\ref{eq310}).
Thus $$a_{k+6}\leq a_{k+5}+a_{k+4}+3b_1+5+2j+l.$$

If $a_{k+6}+b_1=a_{k+5}+a_{k+4}+3b_1+5+2j+l$,  then \begin{eqnarray*}&&P(\{a_1,\ldots,a_{k+6}\})=[0, 2a_{k+5}+2a_{k+4}+6b_1+10+4j+2l] \\
&\setminus& \{b_1, a_{k+5}+a_{k+4}+3b_1+5+2j+l, 2a_{k+5}+2a_{k+4}+5b_1+10+4j+2l\}.\end{eqnarray*}
Thus \begin{eqnarray*}&&a_{k+7}+P(\{a_1,\ldots,a_{k+6}\})=\left[a_{k+7}, a_{k+7}+2\sum_{j=4}^5a_{k+j}+6b_1+10+4j+2l\right]\\
&&\Bigg\backslash \left\{a_{k+7}+b_1, \sum\limits_{\substack{j=4\\j\neq 6}}^7a_{k+j}+3b_1+5+2j+l, a_{k+7}+2\sum_{j=4}^5a_{k+j}+5b_1+10+4j+2l\right\}.\end{eqnarray*}
Similar to the above discussion, we have $$a_{k+7}\leq a_{k+5}+a_{k+4}+3b_1+5+2j+l.$$ Noting that $$a_{k+7}> a_{k+6}=a_{k+5}+a_{k+4}+2b_1+5+2j+l,$$
 $$\sum_{j=4}^5a_{k+j}+3b_1+6+2j+l\leq a_{k+7}+b_1\leq \sum_{j=4}^5a_{k+j}+4b_1+5+2j+l,$$ $$2\sum_{j=4}^5a_{k+j}+5b_1+11+4j+2l\leq a_{k+7}+\sum_{j=4}^5a_{k+j}+3b_1+5+2j+l,$$
$$a_{k+7}+\sum_{j=4}^5a_{k+j}+3b_1+5+2j+l\leq 2\sum_{j=4}^5a_{k+j}+6b_1+10+4j+2l,$$
we have\begin{eqnarray*}P(\{a_1,\ldots,a_{k+7}\})&=&\left[0, a_{k+7}+2\sum_{j=4}^5a_{k+j}+6b_1+10+4j+2l\right] \\&\Big\backslash& \left\{b_1, a_{k+7}+2\sum_{j=4}^5a_{k+j}+5b_1+10+4j+2l\right\}.\end{eqnarray*}
Since $a_{k+4}\geq b_1+4+j+l$, we have $$b_2\geq a_{k+7}+2\sum_{j=4}^5a_{k+j}+5b_1+10+4j+2l>6b_1+10,$$ a contradiction.

If $a_{k+6}+b_1\neq a_{k+5}+a_{k+4}+3b_1+5+2j+l$, then
\begin{eqnarray*}P(\{a_1,\ldots,a_{k+6}\})&=&\left[0, \sum_{j=4}^6a_{k+j}+4b_1+5+2j+l\right]\\
&&\Bigg\backslash \left\{b_1, \sum_{j=4}^6a_{k+j}+3b_1+5+2j+l\right\}.\end{eqnarray*}
Thus $$b_2\geq a_{k+6}+a_{k+5}+a_{k+4}+3b_1+5+2j+l>6b_1+10,$$a contradiction.

{\bf Subcase 2.2:} $a_{k+5}+b_1=a_{k+4}+3b_1+5+2j+l$. By (\ref{eq3.60}) and (\ref{eq3.8}), we have
\begin{eqnarray}\label{eq3.12}&&P(\{a_1,\ldots,a_{k+5}\})=[0, 2a_{k+4}+6b_1+10+4j+2l] \\
&&\setminus \{b_1, a_{k+4}+3b_1+5+2j+l, 2a_{k+4}+5b_1+10+4j+2l\}.\nonumber\end{eqnarray}
Thus \begin{eqnarray}\label{eq3.13}&&a_{k+6}+P(\{a_1,\ldots,a_{k+5}\})=[a_{k+6}, a_{k+6}+2a_{k+4}+6b_1+10+4j+2l]\\
&&\setminus\{a_{k+6}+b_1, a_{k+6}+a_{k+4}+3b_1+5+2j+l, a_{k+6}+2a_{k+4}+5b_1+10+4j+2l\}.\nonumber\end{eqnarray}
If $a_{k+6}>a_{k+4}+3b_1+5+2j+l,$ then
$$b_2=a_{k+4}+3b_1+5+2j+l,$$ which contradicts with (\ref{eq3.7}).
Thus $$a_{k+6}\leq a_{k+4}+3b_1+5+2j+l.$$
Noting that $$a_{k+6}> a_{k+5}=a_{k+4}+2b_1+5+2j+l,$$
$$a_{k+4}+3b_1+6+2j+l\leq a_{k+6}+b_1\leq a_{k+4}+4b_1+5+2j+l,$$
$$2a_{k+4}+5b_1+11+4j+2l \leq a_{k+6}+a_{k+4}+3b_1+5+2j+l\leq 2a_{k+4}+6b_1+10+4j+2l,$$
by (\ref{eq3.12}) and (\ref{eq3.13}) we have \begin{eqnarray*}P(\{a_1,\ldots,a_{k+6}\})&=&[0, a_{k+6}+2a_{k+4}+6b_1+10+4j+2l] \\&&\setminus \{b_1, a_{k+6}+2a_{k+4}+5b_1+10+4j+2l\}.\end{eqnarray*}
Since $a_{k+4}\geq b_1+4+j+l$, we have $$b_2\geq a_{k+6}+2a_{k+4}+5b_1+10+4j+2l>6b_1+10,$$ a contradiction.

This completes the proof of Theorem \ref{thm1}.

\section{Proof of Theorem \ref{thm2}}

Choose $b_1\in\{4,7,8\}\cup \{b: b\geq 11, b\in \mathbb{N}\}$, by the proof of [2, Theorem 1], there exists $A_1=\{a_1<a_2<\cdots<a_k\}\subseteq [1,b_1-1]$ such that
 $P(A_1)=[0,b_1-1].$

Write $b_2=3b_1+5+m$. By the proof of Theorem \ref{thm1}, we have
$$a_{k+2}=b_1+2+j,\quad
a_{k+3}=b_1+3+j+l,$$
where \begin{equation}\label{eq3.140}0\leq j\leq \min\left\{\left\lfloor\frac{m}{2}\right\rfloor, b_1-1\right\}, \quad 0\leq l\leq \min\{m-2j,b_1-1\}.\end{equation}

If $3b_1+5\leq b_2\leq 4b_1+4$, then $0\leq m\leq b_1-1$. By (\ref{eq3.140}) we have
$$0\leq j\leq \left\lfloor\frac{m}{2}\right\rfloor, \quad 0\leq l\leq m-2j.$$
Choose $j=0$ and $l=m$. That is,
 \begin{equation}\label{eq3.14}a_{k+2}=b_1+2, \quad a_{k+3}=b_1+3+m.\end{equation}

If $4b_1+5\leq b_2\leq 4b_1+8$, then $b_1\leq m\leq b_1+3$.
Choose $$j=\left\lceil\frac{m-b_1+1}{2}\right\rceil.$$
Then $m-2j\leq b_1-1$. By (\ref{eq3.140}) we have $0\leq l\leq m-2j$. Let $l=m-2j$. That is,
\begin{equation}\label{eq3.15}a_{k+2}=b_1+2+\left\lceil\frac{m-b_1+1}{2}\right\rceil, \quad a_{k+3}=b_1+3+m-\left\lceil\frac{m-b_1+1}{2}\right\rceil.\end{equation}

 By (\ref{eq3.14}), (\ref{eq3.15}) and Lemma \ref{lem2}(ii), we have \begin{equation}\label{eq317}P(\{a_1,\ldots,a_{k+3}\})=[0, 4b_1+5+m] \setminus \{b_1, 3b_1+5+m\}.\end{equation}
Choose \begin{equation}\label{eq314}a_{i+1}\geq a_i^2\quad (i=k+3,k+4,\ldots).\end{equation}

If $3b_1+5\leq b_2\leq 4b_1+4$, then $$a_{k+4}\geq (b_1+3+m)^2\geq b_1^2+6b_1+9>4b_1+6+m.$$

If $4b_1+5\leq b_2\leq 4b_1+8$, then$$a_{k+4}\geq \left(b_1+3+m-\left\lceil\frac{m-b_1+1}{2}\right\rceil\right)^2> b_1^2+6b_1+9>4b_1+6+m.$$
Hence, if $3b_1+5\leq b_2\leq 4b_1+8$, then we can choose $$a_{k+4}>4b_1+6+m.$$ By (\ref{eq317}), we have $b_3=b_2+b_1+1$.

 Let $A=\{a_1<a_2<\cdots\}$ and $B=\mathbb{N}\setminus P(A)$. By (\ref{eq314}), for $i\geq k+3$ we have
 $$a_{i+1}>a_1+a_2+\cdots+a_i+1.$$
So $a_{i+1}-1\not\in P(A)$ and $a_{i+1}-1\in B$.

If $4b_1+9\leq b_2\leq 6b_1+10$ and $b_2\neq 5b_1+10$, then $b_1+4\leq m\leq 3b_1+5$ and $m\neq 2b_1+5$. Choose $$a_{k+2}=b_1+2, \quad a_{k+3}=b_1+3.$$
By Lemma \ref{lem2}(ii), we have $$P(\{a_1,\ldots,a_{k+3}\})=[0, 4b_1+5] \setminus \{b_1, 3b_1+5\}.$$
Choose $a_{k+4}=m$, we have
\begin{equation}\label{eq4.6}P(\{a_1,\ldots,a_{k+4}\})=[0, 4b_1+5+m] \setminus \{b_1, 3b_1+5+m\}.\end{equation}
Choose \begin{equation}\label{eq30}a_{i+1}\geq a_i^2\quad (i=k+4,\ldots).\end{equation}
By $b_1\geq 4$, we have $$a_{k+5}\geq a_{k+4}^2=m^2\geq b_1^2+8b_1+16>4b_1+6+m,$$ thus by (\ref{eq4.6}) we have $b_3=b_2+b_1+1$.

If $b_2= 5b_1+10$, then $m=2b_1+5$. Choose $j=0,l=1$, that is, $$a_{k+2}=b_1+2, \quad a_{k+3}=b_1+4.$$
By Lemma \ref{lem2}(ii), we have $$P(\{a_1,\ldots,a_{k+3}\})=[0, 4b_1+6] \setminus \{b_1, 3b_1+6\}.$$
Thus
$$a_{k+4}+P(\{a_1,\ldots,a_{k+3}\})=[a_{k+4}, a_{k+4}+4b_1+6] \setminus \{a_{k+4}+b_1, a_{k+4}+3b_1+6\}.$$
Choose $a_{k+4}=2b_1+4=m-1$, we have
\begin{equation}\label{eq4.8}P(\{a_1,\ldots,a_{k+4}\})=[0, 4b_1+5+m] \setminus \{b_1, 3b_1+5+m\}.\end{equation}
Choose \begin{equation}\label{eq33}a_{i+1}\geq a_i^2\quad (i=k+4,\ldots).\end{equation}
By $b_1\geq 4$, we have $$a_{k+5}\geq a_{k+4}^2=(m-1)^2=4b_1^2+16b_1+16>4b_1+6+m,$$ thus $a_{k+5}>4b_1+6+m$. By (\ref{eq4.8}) we have $b_3=b_2+b_1+1$.

 Let $A=\{a_1<a_2<\cdots\}$ and $B=\mathbb{N}\setminus P(A)$. By (\ref{eq30}) and (\ref{eq33}), for $i\geq k+4$ we also have
 $$a_{i+1}>a_1+a_2+\cdots+a_i+1.$$
So $a_{i+1}-1\not\in P(A)$ and $a_{i+1}-1\in B$.

In all cases, we have $b_3=b_2+b_1+1$ and $P(A)=\mathbb{N}\setminus B$.

 This completes the proof of Theorem \ref{thm2}.


\begin{thebibliography}{30}
\bibitem{Burr} S.A. Burr, {\it Combinatorial Theory and its Applications III}. Ed. P. Erd\H{o}s, A. R\'{e}nyi, V.T. S\'{o}s, North-Holland, Amsterdam, 1970.

\bibitem{chen2012} Y.G. Chen and J.H. Fang, {\it On a problem in additive number theory}, Acta Math. Hungar. 134(2012), 416-430.

\bibitem{chen2013} Y.G. Chen and J.D. Wu, {\it The inverse problem on subset sums}, European. J. Combin. 34(2013), 841-845.

\bibitem{Fang2019} J.H. Fang and Z.K. Fang, {\it On an inverse problem in additive number theory}, Acta Math. Hungar. 158(2019), 36-39.

\bibitem{Hegy} N. Hegyv\'{a}ri, {\it On representation problems in the additive number theory}, Acta Math. Hungar. 72(1996), 35-44.

\bibitem{Tang} M. Tang and H.W. Xu, {\it On a problem in additive number theory}, preprint.

\end{thebibliography}
\end{document}